\documentclass{amsart}
\usepackage{amsfonts}
\usepackage{amsmath,amsthm}
\usepackage{amsfonts,amssymb}

\usepackage{enumerate}

\newtheorem{thm}{Theorem}[section]

\newtheorem{lem}[thm]{Lemma}

\newtheorem{rem}[thm]{Remark}

\numberwithin{equation}{section}

\newcommand{\al}{\alpha}

\def\dz{\delta}

\def\sz{\sigma}

\def\da{\delta}

\def\f{\frac}
\def\({\Bigl(}
\def \){ \Bigr)}

 \def\rb{{\mathbf r}}

 \def\RR{{\mathbb R}}

\def\sa{\sigma}
\def\sz{\sigma}

\def\ss{{\Bbb S}^{d}}

\newcommand{\tr}{\triangle}

\def\lb{\langle}
\def\rb{\rangle}

\begin{document}
\def\RR{\mathbb{R}}
\def\Exp{\text{Exp}}
\def\FF{\mathcal{F}_\al}

\title[] {On filtered polynomial approximation on the sphere}

\author{Heping Wang} \address{ School of Mathematical Sciences, Capital Normal
University,
Beijing 100048,
 China.}
\email{wanghp@cnu.edu.cn.}
\author{Ian H. Sloan} \address{School of Mathematics and Statistics, UNSW
Australia, Sydney, NSW, 2052,
Australia}\email{i.sloan@unsw.edu.au}

\date{\today}

\thanks{
 The first author was supported by the
National Natural Science Foundation of China (Project no. 11271263),
 the  Beijing Natural Science Foundation (1132001) and BCMIIS. The second author acknowledges the support of the Australian Research Council through its Discovery program.
 }

\begin{abstract} This paper considers filtered polynomial approximations
on the  unit sphere $\ss\subset \mathbb{R}^{d+1}$, obtained by truncating
smoothly the Fourier series of an integrable function $f$ with the help of
a ``filter'' $h$, which is a real-valued continuous function on
$[0,\infty)$ such that $h(t)=1$ for $t\in[0,1]$ and $h(t)=0$ for $t\ge2$.
The resulting ``filtered polynomial approximation'' (a spherical
polynomial of degree $2L-1$) is then made fully discrete by approximating
the inner product integrals by an $N$-point cubature rule of suitably high
polynomial degree of precision, giving an approximation called ``filtered
hyperinterpolation''. In this paper we require that the filter $h$ and all
its derivatives up to $\lfloor\tfrac{d-1}{2}\rfloor$ are absolutely
continuous, while its right and left derivatives of order $\lfloor
\tfrac{d+1}{2}\rfloor$ exist everywhere and are of bounded variation.
Under this assumption we show that for a function $f$ in the Sobolev space
$W^s_p(\mathbb{S}^d),\ 1\le p\le \infty$, both approximations are of the
optimal order $ L^{-s}$, in the first case for $s>0$ and in the second
fully discrete case for $s>d/p$.

\end{abstract}

\maketitle
\input amssym.def

\section{Introduction}
\label{sec:disNsph.intro}

This paper is concerned with constructive polynomial approximation on the
$d$-dimensional unit sphere $\mathbb{S}^d\subset \mathbb{R}^{d+1}$. Our
ultimate goal is an approximation scheme that uses only function values at
selected points (sometimes called ``standard information''), but as an
intermediate step we will consider also an approximation scheme that
allows inner product integrals (``linear information''). The intermediate
approximation scheme (defined precisely in Subsection \ref{ssec:intro_sub}
below) starts with the Fourier (or Laplace) series for an integrable
function $f$, then truncates it smoothly with a suitable ``filter'' $h$.
The final approximation is then obtained by approximating the inner
product integrals by a cubature rule of suitably high polynomial degree.
Following \cite{SloW}, we shall call the resulting scheme ``filtered
hyperinterpolation''.  We shall call the intermediate approximation
``filtered polynomial approximation''.

Remarkably,  the convergence rate for both of the filtered approximation
schemes is optimal (as explained in Remark 1.2 below or \cite{Tem}), in a
variety of Sobolev space settings. Precisely, under suitable conditions on
the filter $h$ and the cubature rule, and for $s>d/p$ and $1 \le p \le
\infty$, the $L_p$ error of the fully discrete filtered hyperinterpolation
approximation $V_{L,N}$ (a spherical polynomial of degree $2L-1$ defined
in \eqref{1.2} below) satisfies, for a function $f$ in the Sobolev space
$W^s_p(\mathbb{S}^d)$ (see Section \ref{sec:spaces} for definitions of the
Sobolev spaces $W_p^s(\ss)$),
\begin{equation}\label{1.1}
\|f-V_{L,N}(f)\|_{p}\le c L^{-s} \|f\|_{W^s_p},
\end{equation}
where $c$ depends on $p,\ d,\ s$ and the filter $h$ but not on $f$ or $L$
or the cubature rule. The condition $s>d/p$ is the well known necessary
and sufficient condition for $W_p^s(\mathbb{S}^d)$ to be continuously
embedded in $C(\mathbb{S}^d)$, and is unavoidable in any approximation
scheme that employs function values.  The intermediate filtered
approximation scheme has the same rate of convergence, but no requirement
on the smoothness order of the integrand $f$ beyond $s\ge 0$.

The optimal convergence result \eqref{1.1} was proved by Mhaskar \cite{M}
in the context of neural nets for filters of $C^\infty$ smoothness, and by
Sloan and Womersley \cite{SloW} for the case of $p=\infty$ and filters of
smoothness $C^{d+1}$. In the present paper we prove the result \eqref{1.1}
for all values of $p\in [1,\infty]$,  and for filters of smoothness
$W^{\lfloor\frac{d+1}2\rfloor}BV$, where $W^rBV,\ r\in\Bbb N$, denotes the
set of all absolutely continuous functions $h$ whose derivatives of all
orders up to $r-1$ are absolutely continuous and whose right and left
$r$-order derivatives $h^{(r)}_+, h^{(r)}_-$ of $h$ exist everywhere and
are of bounded variation, while $\lfloor a\rfloor$ is the largest integer
not exceeding $a\in\Bbb R$.  For example for the important special case
$d=2$ it is sufficient that $h\in W^1BV$, which certainly holds if $h\in
C^2$, For general $d\ge 2$ we note that
$C^{\lfloor\frac{d+3}2\rfloor}\subset W^{\lfloor\frac{d+1}2\rfloor}BV$ and
$d+1>\lfloor\frac{d+3}2\rfloor$ if $d\ge 2$, so we essentially improve
upon the previous results. Moreover, we prove by counterexample that the
present condition on the smoothness of the filter is  not generally
achieved for filters of smoothness merely
$W^{\lfloor\frac{d-1}2\rfloor}BV$.

The paper \cite{Sloan11} constructed piecewise polynomial filters of
smoothness merely $C^{d-1}$ which generalise the de la
Vall\'ee-Poussin construction for $d=1$, and studied their
properties  for the intermediate semi-discrete approximation and
$p=\infty$. Since these filters lie in $W^dBV$, and since
$d>\lfloor\tfrac{d+1}{2}\rfloor$ for all $d\ge 2$, these filters too
are covered by the present theory.  In the special case $d=2$, even
the de la Vall\'ee-Poussin filter itself (which is simply a
piecewise-linear $C^0$ filter) is covered by the present theory.

\subsection{The approximation schemes}\label{ssec:intro_sub}

The Fourier series for $f\in L_1(\mathbb{S}^d)$ is
\[
\sum_{\ell=0}^\infty\sum_{k=0}^{Z(d,\ell)}  \langle
f,Y_{\ell,k}\rangle Y_{\ell,k}.
\]
Here $\langle \cdot, \cdot\rangle$ denotes the $L_2$ inner product
\begin{equation}\label{1.2-0}
\langle f,g\rangle :=\int_{\mathbb{S}^d}fg d \sigma,
\end{equation}
where $\sigma$ denotes the normalized rotationally invariant  measure  on
$\mathbb{S}^d$, $Z(d,\ell)$ is the dimension of the space of homogeneous
harmonic polynomials on $\mathbb{R}^{d+1}$ of degree $\ell$, and
$(Y_{\ell,k{}})_{\ell=0,1,\ldots;\,k=1,\ldots,Z(d,\ell)}$ is a complete
set of orthonormal spherical harmonics on $\mathbb{S}^d$.

In the filtered polynomial approximation scheme the terms in the
Fourier series are to be modified by multiplication by $h(\ell/L)$,
where $h\in C(\mathbb{R}_+)$ (a ``filter'') satisfies
\[
h(x)=\begin{cases}1\quad \mbox{ if } x\in[0,1],\\
0\quad\mbox{ if } x\ge 2. \end{cases}
\]
Thus the first stage of approximation is
\begin{equation}\label{1.2-1}
f\approx V_{L}(f)\equiv V_L^h(f) :=
\sum_{\ell=0}^\infty\sum_{k=0}^{Z(d,\ell)} h(\tfrac{\ell}{L}) \langle
f,Y_{\ell,k}\rangle Y_{\ell,k}, \end{equation} Note that because of the
properties of the filter we have $V_L(f)\in \Pi^d_{2L-1}$, as well as
$V_L(f)=f$ if $f\in\Pi_L^d$,  where $\Pi^d_L$ is the set of spherical
polynomials on $\mathbb{S}^d$ of degree $\le L$, which is just the set of
polynomials of degree $\le L$ on $\mathbb{R}^{d+1}$ restricted to
$\mathbb{S}^d$. These are properties we want to preserve in the second
stage of approximation.

We obtain the second and final stage of the approximation by discretizing
the inner product \eqref{1.2-0} by an $N$-point cubature rule,
\begin{equation}\label{cub}
\int_{\mathbb{S}^d} g(x)d\sigma(x) \approx Q_N(g):= \sum_{i=1}^N W_i
g(y_i),
\end{equation}
where the points $y_i\in \mathbb{S}^d$ and weights $W_i>0$ are
such that the cubature approximation becomes exact for $g\in
\Pi^d_{3L-1}$, that is such that
\begin{equation}\label{exact}
\int_{\mathbb{S}^d} P(x)d\sigma(x)= Q_N(P)\quad \mbox{for }P\in
\Pi^d_{3L-1}.
\end{equation}
(The $3L-1$ in this condition can be reduced to
$\lceil(2+\delta)L\rceil-1$ for arbitrary but fixed $\delta>0$, but for
simplicity we shall stick with $3L-1$.)  Thus the final approximation is
defined by
\begin{equation}\label{1.2}
V_{L,N}(f)\equiv
V_{L,N}^h(f):=\sum_{\ell=0}^\infty\sum_{k=0}^{Z(d,\ell)}
h(\tfrac{\ell}{L}) \langle f,Y_{\ell,m}\rangle_{Q_N} Y_{\ell,k},
\end{equation}
where $\langle \cdot,\cdot\rangle_{Q_N}$ is the discrete version of
the inner product \eqref{1.2-0}, defined by
\[
\langle f,g\rangle_{Q_N} := Q_N(fg).
\]

Subsequently, the notation $ \ A(N)\asymp B(N) $ means that $A(N)\ll B(N)$
and $A(N)\gg B(N)$, and $A(N)\ll B(N)$ means that there exists a positive
constant $c$ independent of $N$ such that $A(N)\leq cB(N)$, and $A(N)\gg
B(N)$ means $B(N)\ll A(N)$. We write also $a_+:=\max\{a,0\}.$

The main result of the paper can be formulated as follows.

\begin{thm}
Let $h$ be a $W^{\lfloor\frac{d+1}2\rfloor}BV$ filter,  and let
$V_L$ and $V_{L,N}$ be as in \eqref{1.2-1} and \eqref{1.2} with
\eqref{exact} respectively. Then for all $f\in W_p^s(\ss)$, $1\le
p\le \infty$, $s>0$, the inequality
\begin{equation}\label{1.2-2} \|f-V_{L}(f)\|_{p}\ll L^{-s}
\|f\|_{W^s_p}
\end{equation}
holds, and if $s>d/p$ then \eqref{1.1} also holds, that is
\begin{equation}\label{repeat} \|f-V_{L,N}(f)\|_{p}\ll L^{-s}
\|f\|_{W^s_p}.
\end{equation}

\end{thm}

\begin{rem}For $N\in \Bbb N$, the sampling numbers (for optimal recovery) of the Sobolev classes
$BW_p^s(\ss)$ (the unit ball of the space $W_p^r(\ss)$) in $L_p(\ss)$ are
defined by
$$g_N(BW_p^s(\ss),L_p(\ss)):=\inf_{\xi_1,\dots,\xi_N; \varphi}
\sup_{f\in BW_p^s(\ss)}\|f-\varphi(f(\xi_1),\dots,f(\xi_N))\|_p, $$ where
the infimum is taken over all $N$ points $\xi_1,\dots,\xi_N$ in $\ss$ and
all mappings $\varphi$ from $\Bbb R^N$ to $L_p(\ss)$. It follows from
\cite{WW} that
\begin{equation}\label{1.3}
g_N(BW_p^s(\ss),L_p(\ss))\asymp N^{-s/d}.
\end{equation}
To express the error for filtered hyperinterpolation in terms of $N$
rather than $L$, we use the known facts that the minimal number of points
$N$ for a cubature rule with polynomial degree of precision $L$ is
$N\asymp L^d$ (see \cite{Cools97}, Theorem 7.1),  and that cubature rules
with the number of points $N$ of this order do exist (see
\cite{HesseSWHandbook}). For such optimal cubature rules it follows from
\eqref{repeat} that
$$\sup_{f\in BW_p^s(\ss)}\|f-V_{L,N}(f)\|_p\asymp N^{-s/d}\asymp
g_N(BW_p^s(\ss),L_p(\ss)),$$
which implies that the
filtered hyperinterpolation operators $V_{L,N}$ are asymptotically optimal
algorithms in the sense of optimal recovery.

\end{rem}

The paper is organized as follows. Section 2 contains some basic notation
and  properties  concerning harmonic analysis on $\ss$, the operators
$V_L$ and $V_{L,N}$. In Section \ref{sec:proofs} we prove Theorem 1.1. The
proof is modelled on one in \cite{Wang09} for the hyperinterpolation
approximation, where hyperinterpolation is the approximation obtained by
replacing the filter $h$ in \eqref{1.2} by the characteristic function of
$[0,1]$ (see \cite{S}).

\section{Preliminaries}\label{sec:spaces}

Let $\ss=\{x\in\mathbb{R}^{d+1}:\, |x|=1\}$ ($d\ge 2$) be the unit sphere
of $\mathbb{R}^{d+1}$ (with $|x|$ denoting the Euclidean norm in
$\mathbb{R}^{d+1}$) endowed with the rotationally invariant measure $\sz$
normalized by $\int_{\ss}d\sz=1$. For $1\leq p \leq\infty$, denote by $L_p
\equiv L_p(\ss)$ the collection of real measurable functions $f$ on $\ss$
with finite norm
$$\|f\|_p=\Big(\int_{\ss}|f(x)|^p\, d\sz(x)\Big)^\frac{1}{p}<+ \infty,\    \
1\le p<\infty.$$ For $p=\infty,$  the essential supremum is understood
instead of the integral.

  We denote by $\mathcal{H}_\ell^d$ the space of
all spherical harmonics of degree $l$ on $\ss$. It is well known that the
dimension of $\mathcal{H}_\ell^d$ is
$$ Z(d,\ell):={\rm dim}\,\mathcal{H}_\ell^d=\left\{\begin{array}{cl} 1,\ \
\
   & {\rm if}\ \ \ell=0,\\
\frac{(2\ell +d-1)\,(\ell +d-2)!}{(d-1)!\ l!},\ \     & {\rm if}\ \
\ell=1,2,\dots,
\end{array}\right.$$ and$$Z(d,\ell)\asymp l^{d-1},\ \ \ell=1,2,\dots.$$

Let $$\{Y_{\ell,k}\equiv Y_{\ell,k}^{d}\ |\ k=1,\dots, Z(d,\ell)\}$$ be a
fixed orthonormal basis for $\mathcal{H}_\ell^d$. Then
$$\{Y_{\ell,k}\ |\ k=1,\dots, Z(d,\ell),\ \ell=0,1,2,\dots\}$$ is
an orthonormal basis for the Hilbert space $L_2(\ss)$. Thus any $f\in
L_2(\ss)$ can be expressed by its Fourier (or Laplace)
series:$$f=\sum_{\ell=0}^{\infty}
H_\ell(f)=\sum_{\ell=0}^{\infty}{\sum_{k=1}^{Z(d,\ell)} \lb
f,Y_{\ell,k}\rb Y_{\ell,k}}.$$
 where $H_\ell(f)\equiv H_\ell^d(f),\ \ell=0,1,\dots,$ denote the
orthogonal  projections of $f$ onto $\mathcal{H}_\ell^d$, and
$$\lb f,Y_{\ell,k}\rb =\int_{\ss} f(x)
Y_{\ell,k}(x) \, d\sz(x)$$ are the  Fourier coefficients of $f$.
 Let
$C_\ell^\lambda (t)$ be the usual ultraspherical (Gegenbeuer) polynomial
of order ${\lambda}$ normalized by $C_\ell^{\lambda}(1)=\frac{\Gamma(\ell
+2\lambda)}{\Gamma(2\lambda)\Gamma(\ell +1)}$ and  generated by
$$ \frac 1{(1-2zt+z^2)^{\lambda}}=\sum_{k=0}^\infty
C_k^\lambda (t)z^k \ (0\leq z <1).$$ (See (\cite{Sz}, p. 81). The operator
$H_\ell$ can be expressed as follows:
$$H_\ell f(x)={\sum_{k=1}^{Z(d,\ell)}
\lb f,Y_{\ell,k}\rb
Y_{\ell,k}}(x)=\int_{\mathbb{S}^{d}}f(y)K_\ell(x,y)d\sigma(y),$$
where$$K_\ell(x,y)=\sum_{k=1}^{Z(d,\ell)}Y_{\ell,k}(x)Y_{\ell,k}(y)=\frac{\ell
+\lambda}{\lambda}C_\ell^{\lambda}( x\cdot y), \ x,y\in\ss,
\lambda=\frac{d-1}{2},\ d\ge2,$$ which is a ``zonal'' kernel, that is one
depending only on $x\cdot y$.

 It is well known
that the spaces $\mathcal{H}_\ell^d, \ \ell=0,1,2, \dots$, are just the
eigenspaces corresponding to the eigenvalues $-\ell(\ell +d-1)$ of the
Laplace-Beltrami operator $\triangle$ on the sphere. Given $s>0$, we
define the $s$-th order Laplace-Beltrami operator $(-\triangle)^s$ on
$\ss$ in a distributional sense by
$$H_0((-\triangle)^s(f))=0,\ \ \ H_\ell((-\triangle)^s(f))=(\ell(\ell
+d-1))^s H_\ell(f), \  \ \ell=1,2,\dots\ ,$$ where $f$ is a distribution
on $\ss$. For $f\in L_p$, if there exists a function $g\in L_p$ such that
$H_\ell(g)=(\ell(\ell +d-1))^{s/2} H_\ell(f)$ for all $\ell\in\Bbb N$,
then we call $g$ the $s$-th order derivative of $f$ in the sense of $L_p$
and write $g=f^{(s)}=(-\triangle)^{s/2}f$.
  For $s>0$ and $1\le p\le \infty$, the Sobolev space
$W_p^s\equiv W_p^s(\ss)$ is defined by
$$W_p^s:=\Big\{f\in L_p\ :\   \|f\|_{W_p^s}:=\|f\|_p+\|f^{(s)}\|_p<\infty\Big\},$$
while the Sobolev class $BW_p^s$ is defined to be the unit ball of
$W_p^s$. For $f\in L_p, \ 1\le p\le\infty$, we define
$$E_L(f)_p:=\inf\{\|f-P\|_p:P\in\Pi_L^d\},\ \ \
L\in\mathbb{N}.$$ It follows from  the Jackson inequality (see \cite{WL}) that
 $$E_L(f)_p\ll L^{-r}\|f\|_{W_p^s}.$$

Let $h$ be a ``filter", i.e., a continuous function on $[0,\infty)$ such
that $h(t)=1$ if $t\in [0,1]$ and $h(t)=0$ if $t\ge 2$.
 For  $L\geq1$ we define a zonal kernel $\Phi_L$ by
\begin{equation}\label{2.1} \Phi_{L }(x,y)\equiv \Phi_{L }^h(x,y)=\sum_{\ell=0}^{\infty}
 h (\frac lL) K_\ell(x,y), \ \ x,y\in \mathbb{S}^{d},
 \end{equation}
and the operators $V_{L}$  by
\begin{align}V_L(f)(x)&\equiv V_{L}^h(f)(x)=\sum_{\ell=0}^{\infty}{\sum_{k=1}^{Z(d,\ell)}h (\frac{l}{L}) \lb f,Y_{\ell,k}\rb
Y_{\ell,k}}=
\sum\limits_{\ell=0}^{\infty} h (\frac{l}{L})H_\ell(f)(x)\notag \\
&=\int_{\ss}f(y)\Phi_{L }(x,y)d\sz(y),
 \ f\in L_p\ ,\ 1\leq p\leq\infty.\label{2.2}\end{align}
It follows from the definition of $V_L$ that $V_L(f)\in \Pi_{2L-1}^d$ and
$V_L(P)=P$ if $P\in\Pi_L^d$.

For a linear operator $A$ on $L_p(\ss),\ 1\le p\le \infty$, the operator
norm $\|A\|_{(p,p)}$ of $A$ is defined by
$$\|A\|_{(p,p)}:=\sup\{\|Af\|_p\ : f\in L_p(\ss),\ \|f\|_p\le 1\}.  $$
For simplicity we write $\|A\|=\|A\|_{(\infty,\infty)}$. Using  a duality
argument we obtain
\begin{equation}\label{2.3}
\|V_L\|:=\|V_L\|_{(\infty,\infty)}=
\|V_L\|_{(1,1)}=\sup_{x\in\ss}\int_{\ss}|\Phi_{L}(x,y)|d\sz(y)<\infty,
\end{equation}
noting that the last integral is independent of $x$ because of the zonal
nature of $\Phi$. It follows from the Riesz-Thorin theorem that for $1\le
p\le\infty$,
$$\|V_L\|_{(p,p)}\le
\big(\|V_L\|_{(1,1)}\big)^{1/p}\big(\|V_L\|_{(\infty,\infty)}\big)^{1-1/p}=
\|V_L\|<\infty.$$
Since the operators $V_L$ are linear operators on
$L_p(\ss)$ satisfying $V_L(P)=P$ for $P\in\Pi_L^d$, by the Lebesgue
theorem the
 error of approximation by $V_L$ in the $L_p$ norm can be estimated for $s>0$ as
\begin{align}\label{2.5}\|f-V_L(f)\|_p &\le (1+\|V_L\|_{(p,p)})E_L(f)_p\le (1+\|V_L\|)E_L(f)_p\notag\\
 &\ll (1+\|V_L\|) L^{-s}\|f\|_{W_p^s},\qquad\ f\in W_p^s(\ss). \end{align}

The last inequality in \eqref{2.5} shows that estimation of the operator
norm $\|V_L\|$ is extremely important in numerical computation.  We now
turn to this question.

A  function $f$ on $[a,b]$ is of bounded variation and denoted by $f\in
BV[a,b]$,
 if $$\|f\|_{BV}:=\sup_T\sum_{k=1}^{n}|f(x_k)-f(x_{k-1})|<\infty ,$$
where the supremum is taken over all partitions
$T:a=x_0<x_1<\cdots<x_n=b$. As above, for $r\ge 1$ we denote by
$W^rBV[a,b],\ r\in\Bbb N$ the set of all absolutely continuous functions
$h$ on $[a,b]$ for which $h^{(r-1)}$ is absolutely continuous and
$h^{(r)}_+$ and $h^{(r)}_-$ exist and are of bounded variation on $[a,b]$.
It is easy to verify that the function $f(t)=(1-t/2)_+^a,\ a>0$ belongs to
$W^{\lfloor a\rfloor}BV[0,3]$ and does not belong to $C^{\lfloor
a+1\rfloor}[0,3]$. The following lemma gives a condition on $h$ for which
the operator norms $\|V_L\|$ are uniformly bounded, i.e.,
$$\sup_{L\ge1}\|V_L\|<\infty.$$

\begin{lem}Let $h$ be a $W^{\lfloor\frac{d+1}2\rfloor}BV$ filter and let $V_L$ be given as in \eqref{2.2}. Then we have

(1) for any $f\in L_p(\ss), \ 1\le p\le \infty$,
$V_L(f)\in\Pi_{2L-1}^d$, and $V_L(P)=P$ for all $P\in\Pi_L^d$;

(2) the operator norms $\|V_L\|$ are uniformly bounded;

(3) for  $f\in L_p(\ss), \ 1\le p\le \infty$,
$$\|V_L(f)\|_p\ll \|f\|_p;$$

(4) for any $f\in L_p(\ss), \ 1\le p\le \infty$,
$$\|f-V_L(f)\|_p\ll
E_L(f)_p;$$

 (5) for any $f\in W_p^s(\ss), \ 1\le p\le \infty$, $s>0$,
$$\|f-V_L(f)\|_p\ll L^{-s}\|f\|_{W_p^s}.$$

\end{lem}
\begin{proof}It suffices to prove (2). Given $\da>-1$, the Ces\`aro   $(C,\da)$-operators  of the
spherical harmonic  expansion   are defined by
\begin{equation*}\sa_L^\da (f) = \f 1{ A_L^\da} \sum_{ k=0}^L A_{
L-k}^\da H_k(f),\       \ L=0, 1, 2, \ldots,\end{equation*} where $
A_k^\da = \binom{ k+\da} {k},\    \   k=0, 1, 2,\ldots.$ It is well
known that the operator norms $\|\sa_L^\dz\|$ of the Ces\`aro
operators $\sa_L^\da$ are uniformly bounded if and only if
$\dz>\frac{d-1}2$ (see \cite[p. 55]{WL}).

Let $f$ be a real-valued function defined on $\Bbb R$. For $r = 0,1,2,
\dots$, we define the difference operator $\tr^r$ by $$\tr^0=I, \ \ \tr
f(x)=f(x)-f(x+1),\ \ \tr^rf(x)=\tr(\tr^{r-1} f(x)),$$where $I$ denotes the
identity operator. By induction, it is easy to see that
$$\tr^rf(x)=\sum_{k=0}^r(-1)^{k}\binom rkf(x+k).$$ For a sequence
$\{a_k\}_{k=0}^\infty$,  the difference operator $\tr^ra_k$ is
defined as $\tr^r f(k)$ with $f (k) = a_k,\ k\in \Bbb N$. If $h\in
W^rBV$, then we easily verify that
$$\tr^{r}h(\frac kL)=(-1)^{r}L^{-r}\int_{[0,1]^r}h^{(r)}_+(\frac
{k+u_1+\dots+u_r}L)du_1\cdots du_r.$$

 For $r=\lfloor\frac{d+1}2\rfloor$, the summation by parts formula applied
 to \eqref{2.2} gives
\begin{equation}\label{VLf}
V_L(f)=\sum_{k=0}^\infty\tr^{r+1}h(\frac
kL)A_k^{r}\sz_k^r(f)=\sum_{k=0}^{2L}\tr^{r+1}h(\frac
kL)A_k^{r}\sz_k^r(f).
\end{equation}
(See \cite[pp. 408-409]{DX}.) For a $W^{r}BV$ filter $h$ we have
\begin{align*}&\quad\ \sum_{k=0}^{2L}|\tr^{r+1}h(\frac kL)|=
\sum_{k=0}^{2L}|\tr^{r}h(\frac {k+1}L)-\tr^{r}h(\frac
{k}L)|\\
&= L^{-r}
\sum_{k=0}^{2L}\Big|\int_{[0,1]^r}\left(h^{(r)}_+(\frac
{k+1+u_1+\dots+u_r}L)-h^{(r)}_+(\frac
{k+u_1+\dots+u_r}L)\right)du_1\cdots du_r\Big|\\
&\le L^{-r} \int_{[0,1]^r}\sum_{k=0}^{2L}\Big|h^{(r)}_+(\frac
{k+1+u_1+\dots+u_r}L)-h^{(s)}_+(\frac
{k+u_1+\dots+u_r}L)\Big|du_1\cdots du_r\\&\le
L^{-r}\|h^{(r)}_+\|_{BV}.\end{align*} We note that  for
$r=\lfloor\frac{d+1}2\rfloor$ and $k=0,1,\dots, 2L$,
$$ \|\sz_k^r(f)\|_\infty\ll\|f\|_\infty\ {\rm and}\ \ |A_k^r|\ll
(k+1)^r\ll (2L)^r.$$
It then follows from \eqref{VLf} that
\begin{align*}\|V_L(f)\|_\infty&\le\sum_{k=0}^{2L}|\tr^{r+1}h(\frac
kL)|\,A_k^{r}\,\|\sz_k^r(f)\|_\infty\ll
L^{-r}(2L)^r\|f\|_\infty \ll\|f\|_\infty, \end{align*}
which completes the proof
of Lemma 2.1.
\end{proof}

The following lemma gives an  example of a filter $\tilde h$ of smoothness
merely $W^{\lfloor\frac{d-1}2\rfloor}BV$ such that the corresponding
operator norms $\|V_L^{\tilde h}\|$ are not uniformly bounded. This
implies that the optimal order of convergence is not generally achieved
for filters of smoothness merely $W^{\lfloor\frac{d-1}2\rfloor}BV$. So the
condition on the smoothness $W^{\lfloor\frac{d+1}2\rfloor}BV$ of the
filter $h$ for which $\sup_{L}\|V_L^{ h}\|<\infty$   is tight. On the
other hand, we note that this does not preclude the possible existence of
\textbf{particular} filters of smoothness merely
$W^{\lfloor\frac{d-1}2\rfloor}BV$ for which $\sup_L\|V_L^h\|$ is finite.

\begin{lem}There exists a  filter $\tilde h$ of smoothness
merely $W^{\lfloor\frac{d-1}2\rfloor}BV$ such that
$\sup\limits_L\|V_L^{\tilde h}\|=\infty$. \end{lem}
\begin{proof}
Let $h_0(t):=(1-t/2)^{(d-1)/2}_+$, and let $h_1$ be a $C^\infty$ filter on
$[0,\infty)$ such that $h_1(t)$ if $t\in [0,1]$ and $h_1(t)=0$ if $t\ge
3/2$. We set
$$\tilde h(t):=h_0(t)-h_0(t)h_1(t)+h_1(t)=h_0(t)\left(1-h_1(t)\right)+h_1(t).$$
Obviously, $\tilde h$ satisfies $\tilde h(t)=h_1(t)=1$ if
$t\in[0,1]$ and $\tilde h(t)=h_0(t)$ if $t\ge 3/2$, thus $\tilde{h}$
is a filter, but it is a filter of smoothness merely
$W^{\lfloor\frac{d-1}2\rfloor}BV$.  We note that
$h_2(t):=-h_0(t)h_1(t)+h_1(t)$ belongs to $C^\infty$, so the
corresponding operator norms $\|V_L^{ h_2}\|$ are uniformly bounded.
On the other hand, when $h(t)=(1-t/2)^\dz_+$, the corresponding
operators $V_L^{h}$
 are the Riesz  operators
$R_{2L}^\da$.  By the equivalence (due to Marcel Riesz) of the Riesz and
the Ces\`aro summability methods of order $\da\ge0$ (see \cite{Ge}), we
deduce that the Riesz operators $R_L^\da$ are uniformly bounded if and
only if $\dz>\frac{d-1}2$. It follows that the operator norms
$\|V_L^{h_0}\|$ are not uniformly bounded, so the operator norms
$\|V_L^{\tilde h}\|=\|V_L^{h_0}+V_L^{h_2}\|$ are also not uniformly
bounded.\end{proof}

\

 In the sequel, we always assume that $h$ is a $W^{\lfloor\frac{d+1}2\rfloor}BV$ filter.
  It follows from \eqref{2.3} and Lemma 2.1 that
\begin{equation}\label{2.4}
\|V_L\|=\sup_{x\in\ss}\|\Phi_L(x,\cdot)\|_1\ll 1.
\end{equation}
For $f\in W_p^s(\ss),\ 1\le p\le \infty,\ s>d/p$, we define
$$\tau_0(f)=V_{1}(f), \ \ \tau_r(f)=V_{2^{r}}(f)-V_{2^{r-1}}(f)\ \
{\rm for}\ \ r\geq 1.$$ Then $\tau_r(f)\in\Pi_{2^{r+1}}^d$, and
$\sum\limits_{r=0}^{\infty}\tau_r(f)$ converges to $f$ in the uniform
norm. Furthermore, by \eqref{2.5} and \eqref{2.4} we have
\begin{equation}\label{2.6}
\|\tau_r(f)\|_p\le\|f-V_{2^{r}}(f)\|_p+\|f-V_{2^{r-1}}(f)\|_p\ll 2^{-rs}\|f\|_{W_p^{s}}
\end{equation}
for $s>0$.

For $L\in\Bbb N$, we assume that $Q_N(g):= \sum_{i=1}^N W_i g(y_i)$ is a
positive cubature rule on $\ss$ which is exact for $g\in \Pi_{3L-1}^d$,
i.e., the points $y_i\in \mathbb{S}^d$, weights $W_i>0$, satisfy, for all
$g\in\Pi_{3L-1}^d$,
\begin{equation}\label{2.7}\int_{\mathbb{S}^d} g(x)d\sigma(x) = Q_N(g):= \sum_{i=1}^N
W_i g(y_i).
\end{equation}
The filtered hyperinterpolation operator
$V_{L,N}$ is defined by
\begin{equation*}
V_{L,N}(f)\equiv
V_{L,N}^h(f):=\sum_{\ell=0}^\infty\sum_{k=0}^{Z(d,\ell)}
h(\tfrac{\ell}{L}) \langle f,Y_{\ell,m}\rangle_{Q_N} Y_{\ell,k},
\end{equation*}
where $\langle \cdot,\cdot\rangle_{Q_N}$ is the discrete version of
the inner product, defined by $ \langle f,g\rangle_{Q_N} :=
Q_N(fg).$ According to \eqref{2.1}, we have
\begin{equation}\label{2.8}V_{L,N}(f)(x)=\langle
f,\Phi_L(\cdot,x)\rangle_{Q_N}=\sum_{j=1}^NW_if(y_i)\Phi_L(y_i,x).\end{equation}
Note that for any $P\in\Pi_{L}^{d}$ we have $P(\cdot)\Phi_L(x,
\cdot)\in\Pi_{3L-1}^d$
 for arbitrary fixed $x \in \ss$.  Thus by \eqref{2.7} we get
\begin{align}\label{2.9}P(x)=V_L(P)(x)=\langle P,\Phi_L(\cdot,x)\rangle
=\langle P,\Phi_L(\cdot,x)\rangle_{Q_N}=V_{L,N}(P).\end{align}

We are now ready to prove Theorem 1.1.

\section{Proof of Theorem 1.1}\label{sec:proofs}

In order to prove Theorem 1.1, we need the following lemma.

\begin{lem}
\cite[Theorem 2.1]{Dai2}\label{lem0.1} Suppose that $\Gamma$ is a finite set of $\ss$, and $\{\mu_\omega:\omega\in\Gamma\}$  is a set of positive numbers
 satisfying
$$\sum_{\omega\in\Gamma}\mu_\omega|P(\omega)|^{p_0}\ll \int_{\ss} |P(x)|^{p_0}d\sigma(x),\ P\in\Pi_n^d,$$
for some $0<p_0<\infty$ and some positive integer $n$. If $0<p_1<\infty,\
m\geq n,$ and $P\in\Pi_m^d,$ then
 \begin{equation}\label{3.1}\sum\limits_{\omega\in\Gamma}\mu_{\omega}|P(\omega)|^{p_1}\ll(\frac{m}{n})^{d}
 \int_{\ss}|P(y)|^{p_1} d\sigma(y).\end{equation}
 \end{lem}

{\it Proof of Theorem 1.1.} The inequality \eqref{1.2-2} follows directly
from Lemma 2.1 (5), so it suffices to prove \eqref{repeat}. First we
consider the case $p=\infty$. Equation \eqref{2.9} combined with the
Lebesgue theorem implies that for $f\in W_\infty^s(\ss)$ and $s>0$
\begin{equation}\label{2.10}\|f-V_{L,N}(f)\|_\infty
\le (1+\|V_{L,N}\|)E_L(f)_\infty\ll (1+\|V_{L,N}\|)
L^{-s}\|f\|_{W_\infty^s}. \end{equation} Using the standard argument, we
obtain from \eqref{2.8}
$$\|V_{L,N}\|=\sup_{x\in\ss}\sum_{j=1}^NW_j|\Phi_L(x,y_j)|.$$
By \eqref{2.7}, we know  that the set  $\{y_1,\dots,y_N\}$ satisfies the
condition of Lemma \ref{lem0.1} with $p_0=2$. It then follows from Lemma
\ref{lem0.1} with $p_1=1$ that
$$\|V_{L,N}\|\ll \sup_{x\in\ss}\int_{\ss}|\Phi_L(x,y)|d\sz(y)\ll 1,$$
where in the last step we used \eqref{2.4}. Then \eqref{repeat} for the
case $p=\infty$ follows directly from \eqref{2.10}.

Now we consider the case $1\le p<\infty$. We noted already that for $f\in
W_p^s(\ss),\  s>d/p$, the series $\sum\limits_{r=0}^\infty \tau_r (f)(x)$
converges   to $f$ uniformly in $\ss$. Given $L\ge 0$, let $m$ be the
integer satisfying $2^{m}\le L<2^{m+1}$. Since $\tau_r(f)\in
\Pi_{2^{r+1}}^d$, we get from \eqref{2.9} that for $f\in W_p^s(\ss)$ with
$s>p/d$,
$$f(x)-V_{L,N}(f)(x)=\sum_{r=m}^\infty \big(\tau_r(f)(x)-V_{L,N}(\tau_r(f))(x)\big). $$
It follows that
\begin{align}\|f-V_{L,N}(f)\|_p&\leq
\sum_{r\geq m}\|\tau_r (f)-V_{L,N}(\tau_r( f))\|_p \notag\\
& \leq\sum_{r\geq m}\big(\|\tau_r f\|_p +\|V_{L,N}(\tau_r
(f))\|_p\big).\label{3.2}\end{align} We claim that for $1\le p< \infty$
and $r\ge m$,
\begin{equation}\label{3.3} \|V_{L,N}(\tau_r (f))\|_p\ll \big(\frac
{2^{r+1}}L\big)^{d/p}\|\tau_r(f)\|_p.\end{equation} In fact, for $p=1$ and
$r\ge m$, by \eqref{2.4} we obtain
\begin{align*}\|V_{L,N}(\tau_r (f))\|_1&=\|\sum_{j=1}^NW_j\,\tau_r (f)(y_j)\Phi_{L}(\cdot,y_j)\|_1\\
&\le\sum_{j=1}^NW_j\,|\tau_r (f)(y_j)|\,\|\Phi_{L}(\cdot,y_j)\|_1\\
&\ll
 \sum_{j=1}^NW_j|\tau_r (f)(y_j)|
\\ &\ll\big(\frac{2^{r+1}}L\big)^{d}\|\tau_r (f)\|_1,
\end{align*}
where in last step we used again Lemma 3.1 with $p_1=1$.

For $1<p<\infty$ and $r\ge m$, by the H\"older inequality and \eqref{2.4}
we have
\begin{align*}
&\quad\ \|V_{L,N}(\tau_r (f))\|_{p}^p=\int_{\ss}\Big|\sum_{j=1}^N
W_j\, \tau_r (f)(y_j)\,\Phi_L(x,y_j)\Big|^pd\sigma(x)\\
&\le\int_{\ss}\left(\sum_{j=1}^N W_j\, |\tau_r
(f)(y_j)|\,|\Phi_L(x,y_j)|\right)^pd\sigma(x)
\\&\le\int_{\ss}\Big(\sum_{j=1}^N W_j\,|\Phi_L(x,y_j)|\Big)^{p-1}\Big(\sum_{j=1}^N W_j\, |\tau_r
(f)(y_j)|^p\,|\Phi_L(x,y_j)|\Big)d\sigma(x)
\\&\ll\Big(\max_{x\in\ss}\|\Phi_{L}(x,\cdot)\|_1\Big)^{p-1}\,\int_{\ss}\Big(\sum_{j=1}^N W_j\, |\tau_r
(f)(y_j)|^p\,|\Phi_L(x,y_j)|\Big)d\sigma(x)
\\&\ll\sum_{j=1}^N W_j\, |\tau_r
(f)(y_j)|^p\,\|\Phi_L(\cdot,y_j)\|_1\ll\sum_{j=1}^N W_j\, |\tau_r
(f)(y_j)|^p\\ &\ll\big(\frac{2^{r+1}}L\big)^{d}\|\tau_r (f)\|_p^p,
\end{align*}
where we twice used Lemma 3.1, with $p_1=1$ and $p_1=p$. Thus
\eqref{3.3} is proved.

It then follows from \eqref{3.2}, \eqref{3.3},  and \eqref{2.6} that for
$f\in W_p^s(\ss)$, $1\le p< \infty,$ and $s>d/p,$
\begin{align*}\|f-V_{L,N}(f)\|_p&\ll \sum_{r\geq m}\Big(1+\big(\frac
{2^{r+1}}L\big)^{d/p}\Big)\|\tau_r(f)\|_p\notag\\
&\ll  2^{-md/p}\sum_{r\geq m} 2^{-r(s-d/p)}
\|f\|_{W_p^s}\\&\ll 2^{-ms}\|f\|_{W_p^s}\asymp
L^{-s}\|f\|_{W_p^s},\end{align*} which proves \eqref{repeat} for $1\le
p<\infty$.
 The proof of Theorem 1.1 is now complete.

\end{document}